\documentclass[12pt]{article}
\usepackage{amsmath,amssymb,amsthm}

\title{Another proof of Harer-Zagier formula}
\author{Boris Pittel}

\begin{document}
\maketitle
{\center\small Mathematics Subject Classifications: 05C80, 05C30, 05A16, 05E10, 34E05, 60C05}
{\center\small Keywords: surfaces, chord diagrams, genus, random permutations,  Fourier
transform, irreducible characters, Murnaghan-Nakayama, generating  functions}
\def\si{\par\smallskip\noindent}
\def\bi{\par\bigskip\noindent}
\def\pr{\text{ P\/}}
\def\ex{\text{E\/}}
\def\de{\delta}
\def\eps{\varepsilon}
\def\la{\lambda}
\def\a{\alpha}
\def\be{\beta}
\def\de{\Delta}
\def\sig{\sigma}
\def\ga{\gamma}
\def\part{\partial}
\def\Cal{\mathcal}
\def\var{\text{Var\/}}

\newtheorem{Theorem}{Theorem}[section]
\newtheorem{Lemma}{Lemma}[section]
\newtheorem{Proposition}{Proposition}[section]
\newtheorem{Corollary}{Corollary}[section]
\newtheorem{Conjecture}{Conjecture}[section]
\numberwithin{equation}{section}

\begin{abstract} For a regular $2n$-gon there are $(2n-1)!!$ ways to match and glue the $2n$ sides.
The Harer-Zagier bivariate generating function enumerates the gluings by $n$ and the genus $g$
of the attendant surface and leads to a recurrence equation for the counts of gluings with
parameters $n$ and $g$. This formula was originally obtained by using the multidimensional Gaussian integrals. Soon after Jackson and later Zagier found alternative proofs that used the
symmetric group characters. In this note we give a different, characters-based, proof. Its core
is computing and marginally inverting Fourier transform of the underlying probability measure on $S_{2n}$. Aside from Murnaghan-Nakayama rule for one-hook diagrams, the counting techniques we use are of elementary, combinatorial nature.
\end{abstract}
 \section{Introduction and main results}
 Consider a regular, oriented,  $2n$-gon. There are $(2n-1)!!$ ways to match and glue $2n$-sides  
 observing head-to-tail constraint in each glued pair. Each such gluing produces an one-face map
 on an oriented surface. Let $\eps_g(n)$ denote the total number of gluings resulting in
 a surface of genus $g$. Thirty years ago Harer and Zagier \cite{HarerZagier} discovered that
 \begin{equation}\label{HarZag}
 1+2xy +2\sum_{n=1}^{\infty}\frac{x^{n+1}}{(2n-1)!!}\sum_g \eps_g(n)y^{n+1-2g}=
 \left(\frac{1+x}{1-x}\right)^y;
 \end{equation}
 here $n+1-2g$ is the number of vertices in the map on the surface. Of course, 
 $\eps_g(n)/(2n-1)!!$ is the probability that the uniformly random gluing generates a surface
 of genus $g$.  So, introducing the random variable $G_n$, $n\ge 1$, genus of the random surface,
so that $V_n:=n+1-G_n$ is the number of vertices on the surface map, and setting $G_0=1$, we rewrite \eqref{HarZag} as
 \begin{equation}\label{HZ,prob}
 1+2\sum_{n=0}^{\infty}x^{n+1}\ex\bigl[y^{n+1-2G_n}\bigr]= \left(\frac{1+x}{1-x}\right)^y.
 \end{equation}
 Their  proof used a powerful technique based on multidimensional integrals with
 respect to a Gaussian measure on $\Bbb R^k$. The identity \eqref{HarZag} implied a remarkable 
 $3$-term recurrence for the counts $\eps_g(n)$. A year later Jackson \cite{Jackson} found  a group characters-based derivation of an explicit  formula for those counts, expressed through Stirling numbers
 of both kinds, and used it to prove the recurrence anew. Subsequently Itzykson and Zuber \cite{ItzyksonZuber} reduced the combinatorial calculations in \cite{HarerZagier}. In $1995$ Zagier found another, shorter, 
 proof of \eqref{HarZag} based on group characters, see \cite{Zagier} and also Zagier's Appendix to the book \cite{LandoZvonkin} by Lando and Zvonkin. In $2001$ Lass \cite{Lass} gave a combinatorial 
 derivation based upon the enumeration of arborescences and Euler circuits. 
  
The Harer-Zagier formula was used by Linial  and Nowik \cite{LinialNowik} to obtain a sharp
asymptotic formula for $\ex[G_n]$ and later by Chmutov and Pittel \cite{ChmutovPittel} to prove
that $G_n$ is asymptotically Gaussian with mean $(n-\log n)/2$ and variance $(\log n)/2$. 
Pippenger and Schlech \cite{PippengerSchleich}, Gamburd \cite{Gamburd},  Fleming 
and Pippenger \cite{FlemingPippenger} studied the random surface
obtained by gluing together edge-wise $n$ oriented polygonal disks, all with the same number of sides, $3$ in \cite{PippengerSchleich}, and $k\ge 3$ in \cite{Gamburd}, \cite{FlemingPippenger};
$kn$ needs to be even, of course. In addition to the uniformity of the ``gluing''
permutation $\alpha$, Gamburd also assumed that those oriented $k$-gons were the cycles of the 
permutation $\beta$ chosen {\it independently\/} of $\alpha$ and {\it uniformly at random\/} among all permutations with $k$-long cycles only, rather than of a fixed such permutation $\beta$. Fleming and Pippenger showed that the cyclic structures of the resulting permutations 
$\gamma:=\alpha\beta$ are equidistributed. Gamburd used a Fourier transform-based bound for the total variation distance between two probability measures on a finite group, due to
Diaconis and Shashahani \cite{DiaconisShashahani}, \cite{Diaconis}, to prove that, when $2\,\text{lcm}\{2,k\}\, |\, kn$, $\gamma$ is asymptotically uniform  on the alternating subgroup $A_{kn}$. Fleming and Pippenger used
Gamburd's result to obtain very sharp approximations for the moments of $V_n$, confirming the conjectures made in \cite{PippengerSchleich} for the case $k=3$. 
Thus the number of vertices $V_n$
in the surface obtained by gluing the given discs and the number of cycles of $\gamma$ are
equidistributed. Chmutov and Pittel \cite{ChmutovPittel} extended Gamburd's result to the
general case of $n$ polygons with arbitrary ``circumferences'', adding up to an even $N\to\infty$: $\gamma$
is asymptotically uniform on $A_N$ (on $A_N^c$ resp.) if $N-n$ and $N/2$ are of the same (opposite resp.) parity.

In this note we combine Gamburd's ideas and the Fourier transform
on $S_{N}$ to give a proof of the identity \eqref{HZ,prob}.  Aside from Murnaghan-Nakayama rule, mostly
for a simple case of one-hook diagrams, our argument uses only elementary enumerative techniques.

 \section{Derivation of Harer-Zagier formula}
Given an even $N=2n$, let $\alpha$ and $\beta$ be two independent random permutations of $[N]$ chosen uniformly among the permutations with all cycles of length $2$, and among all $(N-1)!$
unicyclic permutations respectively. Our task is to determine the generating function of the number of cycles of
the random permutation $\gamma:=\alpha\beta$.

The starting point is the Fourier inversion formula for a general probability measure $P$ on $S_N$, Diaconis \cite{Diaconis}:
\begin{equation}\label{geninv}
P(s)=\frac{1}{N!}\sum_{\la\vdash N}f^{\la}\,\text{tr}\bigl(\rho^{\la}(s^{-1})\hat P(\rho^{\la})\bigr);
\end{equation}
here $\la$ is a generic partition of $N$, $\rho^{\la}$ is the irreducible representation of $S_N$
associated with $\la$, $f^{\la}=\text{dim}(\rho^{\la})$, and $\hat P(\rho^{\la})$ is the matrix-valued Fourier 
transform of $P(\cdot)$ evaluated at $\rho^{\la}$, 
\[
\hat P(\rho^{\la})=\sum_{s\in S_N} \rho^{\la}(s) P(s).
\]
Let us evaluate the RHS of \eqref{geninv} for $P=P_{\sigma}$, the probability measure on $S_N$
induced by $\sigma=\prod_{j=1}^k \sigma_j$, where $\sigma_j$ is uniform on a conjugacy 
class $C_j$. Now $P_{\sigma}=P_{\sigma_1}\star P_{\sigma_2}\star\cdots\star P_{\sigma_k}$, 
the convolution of $P_{\sigma_1},\dots,P_{\sigma_k}$. So, by multiplicativity
of the Fourier transform for convolutions,
\[
\hat{P}_{\sigma}(\rho^{\la}) = \prod_{j=1}^k \hat {P}_{\sigma_j}(\rho^{\la}).
\]
Since each $P_{\sigma_j}$ is supported by the single conjugacy class $C_j$, we have
$\hat P_{\sigma_j}(\rho^{\la})=\tfrac{\chi^{\la}(C_j)}{f^{\la}}\,I_{f^{\la}}$, see \cite{Diaconis}. So
\[
\hat{P}_{\sigma}(\rho^{\la}) = \prod_{j=1}^k\hat {P}_{\sigma_j}(\rho^{\la})=
(f^{\la})^{-k}\prod_{j=1}^k\chi^{\la}(C_j)\,I_{f^{\la}},
\]
and  \eqref{geninv} becomes
\begin{equation}\label{3chi}
\begin{aligned}
P_{\sigma}(s)&=\frac{1}{N!}\sum_{\la}(f^{\la})^{-k+1}\,\left(\prod_{j=1}^k\chi^{\la}(C_j)\right)\,\text{tr}
\bigl(\rho^{\la}(s^{-1})I_{f^{\la}}\bigr)\\
&=\frac{1}{N!}\sum_{\la}(f^{\la})^{-k+1}\chi^{\la}(s)\prod_{j=1}^k\chi^{\la}(C_j).
\end{aligned}
\end{equation}
For a special case $s=\text{id}$ \eqref{3chi} becomes
\[
\pr(\sigma=\text{id})=\frac{1}{N!}\sum_{\la}(f^{\la})^{-k+2}\prod_{j=1}^k\chi^{\la}(C_j),
\]
and since the LHS is ${\cal N}(C_1,\dots,C_k)$, the number of ways to write the identity
permutation as the product of elements of $C_1,\dots,C_k$, divided by $\prod_{j=1}^k |C_j|$,
we obtain the $S_N$-version of Frobenius's  identity 
\begin{equation}\label{Frob}
{\cal N}(C_1,\dots,C_k)=\frac{\prod_{j=1}^k|C_j|}{N!}\sum_{\la}(f^{\la})^{-k+2}\prod_{j=1}^k\chi^{\la}(C_j).
\end{equation}
Zagier's proof in \cite{Zagier} used \eqref{Frob} for $k=3$. In our argument we use
\eqref{3chi} for $k=2$ only, heavily relying instead on arbitrariness of $s\in S_N$. For $\sigma_1=
\alpha$, $\sigma_2=\beta$ and $\sigma=\gamma=\alpha\beta$, this equation becomes
\begin{equation}\label{Pgamma=}
P_{\gamma}(s)=\frac{1}{N!}\sum_{\la}(f^{\la})^{-1}\chi^{\la}(s)\chi^{\la}({\cal C}_2)
\chi^{\la}({\cal C}_N);
\end{equation}
${\cal C}_2$ (${\cal C}_N$ resp.)
consists of all $(N-1)!!$ permutations with cycles of length $2$ only (all $(N-1)!$ unicyclic permutations resp.).  

Let ${\cal C}_{\boldsymbol{\nu}}$ be a generic conjugacy class comprising all permutations with given counts $\nu_j$
of cycles of length $j\in [1,N]$. Define $J=J(\boldsymbol\nu)=\{j\in [1,N]: \nu_j>0\}$. Let $\boldsymbol{\alpha}=(\alpha_1,\alpha_2,\dots)$ be an 
{\it arbitrary\/} composition of $N$ containing $\nu_j$ components equal $j$, ($j\in J$). By Murnaghan-Nakayama rule,
Stanley \cite{Stanley} (Section 7.17, Equation (7.75)),
\begin{equation}\label{Mur}
\chi^{\la}({\cal C}_{\boldsymbol{\nu}})=\sum_T(-1)^{\text{ht}(T)},
\end{equation}
where the sum is over all rim hook diagrams $T$ of shape $\la$ and type $\boldsymbol{\alpha}$, 
i. e. over all ways to empty the diagram $\la$ by
successive deletions of the rim hooks, one hook at a time, of lengths $\alpha_1,\alpha_2,\dots$.
Further $\text{ht}(T)$ is the sum of heights of the individual hooks (number of rows minus $1$) in
the hook diagram $T$. 

This remarkable formula implies that $\chi^{\la}({\cal C}_N)=0$ unless the
diagram $\la$ is a single hook $\la^*$, with one row of length $\la_1$ and one column of height $\la^{1}$,
$\la_1+\la^1=N+1$, in which case 
\begin{equation}\label{chiN}
\chi^{\la}({\cal C}_N)=(-1)^{\la^1-1}. 
\end{equation}
Consider $\chi^{\la^*}({\cal C}_2)$. If $\la_1$ is even, then $\la^1$ is odd, so that for each $T$ the 
rim hook deleted last consists of two first cells in the row.
So $\text{ht}(T)=(\la^1-1)/2$ for every $T$, and the number of $T$'s is the number of ways
to intersperse {\it first\/} $(\la_1-2)/2$  deletions, from left to right,  of domino tiles in the row with
$(\la^1-1)/2$  deletions of domino tiles, from bottom to top, in the column, and this number is
\begin{equation}\label{la1even}
\binom{\tfrac{\la_1-2+\la^1-1}{2}}{\tfrac{\la_1-2}{2}}=\binom{\tfrac{N-2}{2}}{\tfrac{\la_1-2}{2}}\Longrightarrow \chi^{\la^*}({\cal C}_2)=(-1)^{(\la^1-1)/2}\binom{\tfrac{N-2}{2}}{\tfrac{\la_1-2}{2}}.
\end{equation}
Analogously, if $\la_1$ is odd then
\begin{equation}\label{la1odd}
\chi^{\la^*}({\cal C}_2)=(-1)^{\la^1/2}\binom{\tfrac{N-2}{2}}{\tfrac{\la_1-1}{2}}.
\end{equation}
Combining \eqref{chiN}-\eqref{la1odd}, we have
\begin{equation}\label{chiCNxchiC2}
\begin{aligned}
\chi^{\la^*}({\cal C}_2)\chi^{\la^*}({\cal C}_N)&=(-1)^{(\la^1-1)/2}\binom{\tfrac{N-2}{2}}{\tfrac{\la_1-2}{2}},\quad\text{if }\la_1\text{ is even};\\
\chi^{\la^*}({\cal C}_2)\chi^{\la^*}({\cal C}_N)&=(-1)^{(\la^1+2)/2}\binom{\tfrac{N-2}{2}}{\tfrac{\la_1-1}{2}},\quad\,\text{if }\la_1\text{ is odd}.
\end{aligned}
\end{equation}
As for $f^{\la^*}$, applying the Hook formula we obtain
\begin{equation}\label{fla*=}
f^{\la^*}=\frac{N!}{N\prod_{r=1}^{\la_1-1} r\,\prod_{s=1}^{\la^1-1}s}=\binom{N-1}{\la_1-1}.
\end{equation}
Introducing
\begin{equation}\label{defF}
\begin{aligned}
F(N,\la_1)&:=(-1)^{(\la^1-1)/2}\binom{\tfrac{N-2}{2}}{\tfrac{\la_1-2}{2}}/\binom{N-1}{\la_1-1}\quad\text{if }\la_1\text{ is even};\\
F(N,\la_1)&:=(-1)^{(\la^1+2)/2}\binom{\tfrac{N-2}{2}}{\tfrac{\la_1-1}{2}}/\binom{N-1}{\la_1-1}\quad\text{if }\la_1\text{ is odd},
\end{aligned}
\end{equation}
we transform \eqref{geninv} into
\begin{equation}\label{specinv}
P_{\gamma}(s)=\frac{1}{N!}\sum_{\la_1=1}^NF(N,\la_1)\chi^{\la^*}(s);
\end{equation}
here $\la^*$ is a hook of size $N$, with the row of length $\la_1$, and the column of height
$\la^1=N+1-\la_1$.

Turn  to  $\chi^{\la^*}(s)$. Let  $s\in {\cal C}_{\boldsymbol{\nu}}$. Introduce $\ell=\ell(\boldsymbol\nu):=\min J=\min\{j:\nu_j>0\}$, the size of the smallest positive component of $\alpha$.

{\it Case\/} $\la_1\le\ell$. Here, for every $\boldsymbol\alpha$, there is only one rim hook tableau $T$, since the arm of the hook deleted last must be the whole row of $\la^*$, and, denoting $\nu=\sum_j\nu_j$, 
\begin{equation}\label{Case1}
\text{ht}(T)=\la^1-\nu\equiv (\la^1+\nu)\,(\text{mod } 2)\Longrightarrow \chi^{\la^*}(s)=(-1)^{\la^1+\nu}.
\end{equation}

{\it Case\/} $\la_1> \ell$. Since in \eqref{Mur}
the composition $\boldsymbol{\alpha}$ can be chosen arbitrarily, let us assume
that components of $\boldsymbol\alpha$ are non-increasing. So the composition
$\boldsymbol\alpha$ consists of the segment formed by
all components of the largest size, followed  by the segment formed by
all components of the second largest size, and so on, all the way to the terminal segment formed
by all components of the smallest size $\ell$.  Therefore in a generic 
tableau $T$ the last rim hook is a diagram $\mu^*$ comprising one row and one column of sizes $\mu_1$ and $\mu^!$, $\mu_1+\mu^1=\ell+1$. All the other rim hooks in $T$ are either
horizontal or vertical, successively deleted from the row and from the column of $\la^*$ respectively. Let 
$h_r$ be the number of those horizontal hooks of size $r\in J$; so $h_r\in [0,\nu_r]$ for $r>\ell$,
and $h_{\ell}\in [0,\nu_{\ell}-1]$. The admissible $\bold h=\{h_r\}$ must meet the additional constraint
\begin{equation}\label{+mu1=la1}
\sum_{r\in J} h_r r +\mu_1= \la_1.
\end{equation}
The total number of the tableaux $T$, with parameters $\mu_1,\,\mu^1,\,\bold h$, is
\[
\binom{\nu_{\ell}-1}{h_{\ell}}\prod_{r\in J\setminus\{\ell\}}\binom{\nu_r}{h_r}=
\prod_{r\in J}\binom{\nu_r-\delta_{\ell,r}}{h_r},\qquad\delta_{\ell,r}:=1_{\{r=\ell\}}.
\]
Further, the $\mu^1$-long leg of the last hook contributes $\mu^1-1$ to the height $\text{ht}(T)$,  while the total contribution to $\text{ht}(T)$ of those vertical rim hooks is $(\la^1-\mu^1)$, the sum of their sizes,  minus 
$\sum_r(\nu_r-\delta_{\ell,r}-h_r)$, their total number. So
\begin{align*}
\text{ht}(T)&=(\mu^1-1)+(\la^1-\mu^1) -\sum_{r\in J}\, (\nu_r-\delta_{\ell,r}-h_r)\\
&\equiv \la^1-1+\sum_{r\in J}\,(\nu_r-\delta_{\ell,r}-h_r)\,(\text{mod }2).
\end{align*}
Therefore the total contribution to $\chi^{\la^*}(s)$ of the rim hook tableaux $T$ with the last rim hook $\mu^*$ is
\begin{align*}
&(-1)^{\la^1-1}\sum_{\bold h\text{ meets }\eqref{+mu1=la1}}\,\prod_{r\in J} (-1)^{\nu_r-\delta_{\ell,r}-
h_r}\binom{\nu_r-\delta_{\ell,r}}{h_r}\\
&=(-1)^{\la^1-1} [\xi^{\la_1-\mu_1}] \prod_{r\in J}\sum_{h_r}(-1)^{\nu_r-\delta_{\ell,r}-h_r}(\xi^r)^{h_r}
\binom{\nu_{\ell}-\delta_{\ell,r}}{h_r}\\
&=(-1)^{\la^1-1} [\xi^{\la_1-\mu_1}] \prod_{r\in J}(\xi^r-1)^{\nu_r-\delta_{\ell,r}}.
\end{align*}
To get $\chi^{\la^*}(s)$ we need to sum this expression for all $1\le \mu_1\le \min\{\ell,\lambda_1\}$,
\linebreak i. e. for $1\le\mu_1\le \ell$, because  $\lambda_1>\ell$. For those $\mu_1$, $\la_1-\mu_1$ ranges from $\la_1-\ell$ to $\la_1-1$. Thus
\begin{equation}\label{Case2}
\begin{aligned}
\chi^{\la^*}(s)&=(-1)^{\la^1-1}\sum_{t=\la_1-\ell}^{\la_1-1}[\xi^t]\prod_{r\in J}(\xi^r-1)^{\nu_r-\delta_{\ell,r}}\\
&=(-1)^{\la^1-1}\,[\xi^{\la_1-1}] \left(\sum_{\tau=0}^{\ell-1}\xi^{\tau}\right)\prod_{r\in J}(\xi^r-1)^{\nu_r-\delta_{\ell,r}}\\
&=(-1)^{\la^1-1}\,[\xi^{\la_1-1}] \,\frac{\xi^{\ell}-1}{\xi-1}\prod_{r\in J}(\xi^r-1)^{\nu_r-\delta_{\ell,r}}\\
&=(-1)^{\la^1-1}\,[\xi^{\la_1-1}]\, (\xi-1)^{-1}\prod_{r\in J}(\xi^r-1)^{\nu_r}\\
&=(-1)^{\la^1+\nu}\,[\xi^{\la_1}]\, \frac{\xi}{1-\xi}\prod_{r\in J}(1-\xi^r)^{\nu_r}.
\end{aligned}
\end{equation}
Observe that, for $\la_1\le \ell$, we have $r>\la_1-1$ for all $r\in J$; so the bottom RHS in  \eqref{Case2} is
\begin{multline*}
(-1)^{\la^1+\nu}[\xi^{\la_1-1}]\, (1-\xi)^{-1}\prod_{r\in J}(\xi^r-1)^{\nu_r}\\
=(-1)^{\la^1+\nu}[\xi^{\la_1-1}]\, (1-\xi)^{-1}=(-1)^{\la^1+\nu},
\end{multline*}
which is the value of $\chi^{\la^*}(s)$ for $\la_1<\ell$, see \eqref{Case1}. 
Therefore, the bottom line expression in \eqref{Case2} for $\chi^{\la^*}(s)$ holds for all $\la_1$.

Though seemingly unwieldy, the formula \eqref{specinv} for $\pr(\gamma=s)$ together with 
\eqref{Case2}  instantly lead to a promising, intermediate, expression for the {\it marginal\/} distribution of $X_n$, the number of cycles of $\gamma$.
\begin{Theorem}\label{P(XN)=} For $\nu\in [1,N]$, $N=2n$,
\begin{equation}\label{P(XN)=expl}
\pr(X_n=\nu)=[x^Ny^{\nu}] \sum_{\la_1=1}^N\!(-1)^{\la^1}F(N,\la_1)
[\xi^{\la_1}]\!\left[\frac{\xi}{1-\xi}\left(\frac{1-x}{1-x\xi}\right)^y\right],
 \end{equation}
or equivalently
\begin{equation}\label{gen,nu}
\ex[y^{X_n}]=\sum_{\la_1=1}^N(-1)^{\la^1}F(N,\la_1)
\cdot [\xi^{\la_1}x^N]\!\left[\frac{\xi}{1-\xi}\left(\frac{1-x}{1-x\xi}\right)^y\right].
\end{equation}
\end{Theorem} 
\begin{proof} By \eqref{specinv},
\[
\pr(X_n=\nu)=\frac{1}{N!}\sum_{\la_1=1}^NF(N,\la_1)\!\!\sum_{s\,: \,\nu(s)=\nu}\!\!
\chi^{\la^*}(s).
\]
Here, by \eqref{Case2} and Cauchy formula 
$
|{\cal C}_{\boldsymbol\nu}|=N!\prod_r\tfrac{1}{r^{\nu_r}\nu_r!},
$
we obtain
\begin{align*}
&\sum_{s: \boldsymbol{\nu}(s)=\boldsymbol\nu}\chi^{\la^*}(s)=(-1)^{\la^1}N!\,[\xi^{\la_1}]
\frac{\xi}{1-\xi}
\sum_{\sum_r\nu_r=\nu,\atop \sum_r r\nu_r=N}\,\,\prod_{r\ge 1}\left(\frac{-\,(1-\xi^r)}{r}\right)^{\nu_r}
\!\!\!\!\!\big/\nu_r!\\
&=(-1)^{\la^1}N!\,[\xi^{\la_1}]\frac{\xi}{1-\xi}\,[x^Ny^{\nu}]\prod_{r\ge 1}\left[\,\sum_{\nu_r\ge 0}
\left(\frac{-\,yx^r(1-\xi^r)}{r}\right)^{\nu_r}\!\!\!\big/\nu_r!\right]\\
&=(-1)^{\la^1}N!\,[\xi^{\la_1}]\frac{\xi}{1-\xi}\,[x^Ny^{\nu}]\exp\left(\!-\sum_{r\ge 1}\frac{yx^r(1-\xi^r)}{r}
\right)\\
&=(-1)^{\la^1}N!\,[\xi^{\la_1}]\frac{\xi}{1-\xi}\,[x^Ny^{\nu}]\exp\left(\!-y\log\frac{1}{1-x}+y\log\frac{1}{1-x\xi}
\right)\\
&=(-1)^{\la^1}N!\,[\xi^{\la_1}]\frac{\xi}{1-\xi}\,[x^Ny^{\nu}]\left(\frac{1-x}{1-x\xi}\right)^y.
\end{align*}
The proof  is complete.
\end{proof}
Let us use \eqref{gen,nu} to get an explicit formula for $\pr(X_n=\nu)$. To begin, from  \eqref{defF} 
it follows that, both for $\la_1=2m$, ($0<m\le n$) and $\la_1=2m+1$, ($0\le m<n-1$),
\begin{equation}\label{Q(n,m)=}
\begin{aligned}
(-1)^{\la^1}F(N,\la_1)&=(-1)^{n-m+1}\frac{\binom{n-1}{m-1}}{\binom{2n-1}{2m-1}}
=(-1)^{n-m+1}Q(n,m),\\
Q(n,m)&:=\frac{(2m-1)!!\,\bigl(2(n-m)-1\bigr)!!}{(2n-1)!!},
\end{aligned}
\end{equation}
with $(-1)!!:=1$.  Next we evaluate $A(n,\nu,\la_1)$, the coefficient of $x^Ny^{\nu}\xi^{\la_1}$ in the
expansion of $\tfrac{\xi}{1-\xi}\left(\tfrac{1-x}{1-x\xi}\right)^y$. It is well known, Comtet \cite{Comtet} (Section 5.5),  that
\begin{equation}\label{wellknown}
\frac{1}{\nu!}\left(\log\frac{1}{1-\eta}\right)^{\nu}=\sum_{\ell\ge \nu}\eta^{\ell}\,\frac{s(\ell,\nu)}{\ell!},
\end{equation}
where $s(\ell,\nu)$ is the (first kind Stirling) number of permutations of $[\ell]$ with $\nu$ cycles. So, setting 
\[
\frac{1-x}{1-x\xi}=\frac{1}{1-\eta}\Longleftrightarrow \eta=\frac{x(\xi-1)}{1-x},
\]
we have
\begin{align*}
[y^{\nu}]\left(\frac{1-x}{1-x\xi}\right)^y&=\frac{1}{\nu!}\left(\frac{1}{1-\eta}\right)^{\nu}
=\sum_{\ell\ge\nu}\frac{s(\ell,\nu)}{\ell!}\left(\frac{x(\xi-1)}{1-x}\right)^{\ell}.
\end{align*}
Next, for $\ell\le 2n$,
\begin{align*}
[x^{2n}]\left(\frac{x(\xi-1)}{1-x}\right)^{\ell}=[x^{2n-\ell}]\left(\frac{\xi-1}{1-x}\right)^{\ell}
=(\xi-1)^{\ell}\binom{-\ell}{2n-\ell}(-1)^{2n-\ell},
\end{align*}
and finally
\[
[\xi^{\la_1}]\frac{\xi}{1-\xi}(\xi-1)^{\ell}=-[\xi^{\la_1-1}](\xi-1)^{\ell-1}=(-1)^{\ell-\la_1+1}\binom{\ell-1}{\la_1-1}.
\]
Collecting the pieces we arrive at
\begin{equation}\label{collect}
\begin{aligned}
A(n,\nu,\la_1)&=
(-1)^{\la_1-1}\sum_{\ell\ge \nu}\frac{s(\ell,\nu)}{\ell!}\binom{-\ell}{2n-\ell}\binom{\ell-1}{\la_1-1}\\
&=(-1)^{\la_1-1}\binom{2n-1}{2n-\la_1}\sum_{\ell\ge\nu}(-1)^{\ell}\,\frac{s(\ell,\nu)}{\ell!}\binom{2n-\la_1}{2n-\ell}.
\end{aligned}
\end{equation}
Let $P_e(X_n=\nu)$ and $P_o(X_n=\nu)$ denote, respectively,  the contribution of even $\la_1$  and odd $\la_1$ to the RHS of \eqref{P(XN)=expl}. To compute $P_e(X_n=\nu)$ and $P_o(X_n=\nu)$ we will need two simple identities
\begin{align*}
\sum_{j=0}^a(-1)^j\binom{a}{j}\binom{2j}{b}&=(-1)^a\binom{a}{b-a}2^{2a-b},\\
\sum_{j=0}^a(-1)^j\binom{a}{j}\binom{2j+1}{b}&=(-1)^a\left[\binom{a}{b-a}2^{2a-b}+\binom{a}{b-a-1}
2^{2a+1-b}\right],
\end{align*}
directly implied by
\begin{align*}
\sum_{b\ge 0}x^b\sum_{j=0}^a(-1)^j\binom{a}{j}\binom{2j}{b}&=[1-(1+x)^2]^a=(-1)^a x^a(2+x)^a,\\
\sum_{b\ge 0}x^b\sum_{j=0}^a(-1)^j\binom{a}{j}\binom{2j+1}{b}&=(-1)^a\bigl(x^a+x^{a+1}\bigr)(2+x)^a.
\end{align*}
(see Gould \cite{Gould} or  www.math.wvu.edu/gould/, Vol.4.PDF, (1.62)).
By \eqref{collect}, and \eqref{Q(n,m)=},
\begin{equation}\label{Pe=}
\begin{aligned}
P_e(X_n=\nu)&=\sum_{m=1}^n(-1)^{n-m}Q(n,m)\binom{2n-1}{2(n-m)}\\
&\,\,\,\,\,\times\sum_{\ell\ge \nu}
(-1)^{\ell}\,\frac{s(\ell,\nu)}{\ell!}\binom{2(n-m)}{2n-\ell}\\
&=\sum_{\ell\ge \nu}
(-1)^{\ell}\,\frac{s(\ell,\nu)}{\ell!}\sum_{j=0}^n(-1)^{j}\binom{n-1}{j}\binom{2j}{2n-\ell}\\
&=\sum_{\ell\ge\nu}(-1)^{\ell+n-1}2^{\ell-2}\frac{s(\ell,\nu)}{\ell!}\binom{n-1}{\ell-2}.
\end{aligned}
\end{equation}
Similarly
\begin{equation}\label{Po=}
\begin{aligned}
P_o(X_n=\nu)&=
\sum_{\ell\ge \nu}
(-1)^{\ell}\,\frac{s(\ell,\nu)}{\ell!}\sum_{j=0}^{n-1}(-1)^{j}\binom{n-1}{j}\binom{2j+1}{2n-\ell}\\
&=\sum_{\ell\ge\nu}(-1)^{\ell+n-1}\frac{s(\ell,\nu)}{\ell!}\left[2^{\ell-2}\binom{n-1}{\ell-2}+2^{\ell-1}
\binom{n-1}{\ell-1}\right].
\end{aligned}
\end{equation}
Combining \eqref{Pe=} and \eqref{Po=} we have
\begin{equation}\label{P=}
\begin{aligned}
\pr(X_n=\nu)&=\sum_{\ell\ge\nu}(-1)^{\ell+n-1}2^{\ell-1}\frac{s(\ell,\nu)}{\ell!}\left[\binom{n-1}{\ell-2}
+\binom{n-1}{\ell-1}\right]\\
&=\sum_{\ell\ge\nu}(-1)^{\ell+n-1}2^{\ell-1}\frac{s(\ell,\nu)}{\ell!}\binom{n}{\ell-1}.
\end{aligned}
\end{equation}
(This simple formula appears to be new.) Therefore
\begin{align*}
\pr(X_n=\nu)&=\frac{1}{2}\sum_{\ell\ge\nu}2^{\ell}\,\frac{s(\ell,\nu)}{\ell!}\,[x^{n+1-\ell}](1+x)^{-\ell}\\
&=\frac{1}{2}[x^{n+1}]\sum_{\ell\ge\nu}\frac{s(\ell,\nu)}{\ell!}\left(\frac{2x}{1+x}\right)^{\ell}
=\frac{1}{2}[x^{n+1}]\frac{1}{\nu!}\left(\log\frac{1+x}{1-x}\right)^{\nu},
\end{align*}
by \eqref{wellknown}, as $\tfrac{2x}{1+x}=1-\tfrac{1-x}{1+x}$. Consequently
\[
\ex[y^{X_n}]=\frac{1}{2}[x^{n+1}]\sum_{\nu\ge 1}\frac{1}{\nu!}\left(\log\frac{1+x}{1-x}\right)^{\nu}
=\frac{1}{2}[x^{n+1}]\left[\left(\frac{1+x}{1-x}\right)^y-1\right],
\]
whence, setting $X_0=1$,
\[
\sum_{n\ge 0}x^{n+1}\ex[y^{X_n}]=\frac{1}{2}\left[\left(\frac{1+x}{1-x}\right)^y-1\right] 
\leftrightarrow 
1+2\sum_{n\ge 0}x^{n+1}\ex[y^{X_n}]=\left(\frac{1+x}{1-x}\right)^y,
\]
which is the Harer-Zagier formula for the genus $G_n$, as $X_n=n+1-2G_n$.\\

\noindent {\bf Acknowledgment.\/}  I owe a debt of gratitude to Huseyin Acan for innumerable
discussions of the random chord diagrams. It is my pleasure to thank  Sergei Chmutov and Sergei Duzhin for introducing me to  the remarkable Harer-Zagier formula. I am
grateful to Sergei Chmutov for patiently explaining to me the topological issues  of the gluing models, and for a very helpful feedback on an initial draft of this paper.

\end{document}